%% file: ShortSMA.tex
\documentclass[11pt]{cpres}

\newcommand{\vol}{\operatorname{vol}}

\title{Addendum to: Guts, Volume and Skein Modules of 3-Manifolds}

\begin{document}
\author{Brandon Bavier}
\address[]{Department of Mathematics, Michigan State University, East
Lansing, MI, 48824, USA}
\email[]{bavierbr@msu.edu}

\maketitle
\begin{abstract}
In \textit{Guts, Volume, and Skein Modules of 3-Manifolds} \cite{GVSM}, we showed that the twist number of certain hyperbolic weakly generalized alternating links can be recovered from a Jones-like polynomial, and offers a lower bound for the volume of the link complement. Here, we modify the proof to work for a larger class of links.


\end{abstract}

\section{Introduction}

The goal of this paper is to generalize the results of our recent paper \cite{GVSM}, which showed that, under certain hypothesis, the volume of a hyperbolic link in a compact, irreducible 3-manifold $M$ that admits an alternating projection on a closed surface $F\subset M$, is bounded below in terms of the Kauffman bracket function defined on link diagrams on $F$. In the paper, they generalize the proofs and results of Dasbach and Lin \cite{VTJPAK}, as well as Futer, Kalfagianni, and Purcell \cite{Guts,DFVJP} to a subset of weakly generalized alternating knots, as defined by Howie and Purcell \cite{WGA}.

In particular, in \cite{GVSM}, we worked with hyperbolic links that admit a prime, twist-reduced alternating projection diagram $D$ on a surface $F$ in a manifold $M$, with the additional properties that any essential curve of $F$ intersects our knot at least three times and that $F\setminus D$ is checkerboard colorable. In this paper, we will generalize our result to all links admitting twist-reduced, reduced alternating projections. We do this by looking at additional coefficients coming from the bracket function. We will recall all these definitions in Section 2.

Let $M$ be an irreducible, compact 3-manifold with or without boundary, and let $F\subset M$ be a closed, orientable, embedded surface. Then, if we have a link $L\subset F\times [-1,1]$, we can get a projection under the map $\pi: F\times[-1,1]\to F=F\times\{0\}$. We can then define, for connected $F$, a Kauffman bracket function, and construct a family of Jones-like polynomials
\[J_X(\pi(L)) = a_{m,X}t^m + a_{m-1,X}t^{m-1} + \cdots + b_{n+1,X}t^{n+1} + b_{n,X}t^n,\]
where $X$ is a collection of simple disjoint closed curves on $F$. We will give the precise definitions for these polynomials in the next section. We can consider the sums of these coefficients,

\[a_k = \sum_{X} a_{k,X} \qquad b_k = \sum_X b_{k,X},\]
where we sum over all possible collections of simple disjoint closed curves on $F$. Then we are able to say the following:

\begin{thm}
\label{thm:altVol}
Let $M$ be an irreducible, compact 3-manifold with empty or incompressible boundary.
Let $F\subset M$ be an incompressible, closed, orientable surface such that $M\setminus N(F)$ is atoroidal and $\bndry-anannular$.
Suppose that a link $L$ admits a weakly generalized alternating projection $\pi(L)\subset F$ that is reduced, twist-reduced and with all regions of $F\setminus \pi(L)$ disks.
Finally, suppose that $r(\pi(L),F)>4$. Then $L$ is hyperbolic and
\[\vol(M\setminus L)\ge v_8\mathrm{max}\{|a_{m-1}|, |b_{n+1}|\}-\frac{1}{2}\chi(\bndry M),\]
where $v_8=3.66386\ldots$ is the volume of an ideal octahedron.
\end{thm}

Just as in \cite{GVSM}, we are able to get the following result involving the Jones-like polynomial, the twist number, and the guts of our link complement:

\begin{thm}
\label{thm:altTwst}
Let $M$ and $F$ be as in Theorem \ref{thm:altVol} 
Suppose that a link $L$ admits a weakly generalized alternating projection $\pi(L)\subset F$ that is twist-reduced and with all regions of $F\setminus \pi(L)$ disks. Then we have the following.
\begin{enumerate}
	\item $\chi(\mathrm{guts}(M_A)) = -|a_{m-1}| + \frac{1}{2}\chi(M)$
	\item $\chi(\mathrm{guts}(M_B)) = -|b_{n+1}| + \frac{1}{2}\chi(M)$
	\item $t_F(\pi(L)) = |a_{m-1}| + |b_{n+1}| + \chi(F)$.
\end{enumerate}
Further, if $\pi(L)$ is a cellularly embedded, twist-reduced, reduced alternating diagram, we still get (3).
\end{thm}

As in \cite{GVSM}, Theorem \ref{thm:altVol} follows directly from \ref{thm:altTwst}. Also, as in \cite{GVSM}, we get the following result for links in thickened surfaces.

\begin{cor}
Let $L$ be a link in $F\times[-1,1]$ that admits a reduced diagram $\pi(L)\subset F$ that is twist reduced and has $F\setminus \pi(L)$ disk regions. Then any two diagrams for $L$ have the same twist number. That is, $t_F(L)$ is an isotopy invariant of $L$.
\end{cor}

The remainder of this paper is split into two sections. First, in Section 2, we recall some key definitions related to generalized link diagrams, including the weakly generalized alternating diagrams of Howie and Purcell \cite{WGA} and reduced diagrams. Here, we will also introduce the bracket function and Jones-like polynomial we use in our results. Then, in Section 3, we prove the Theorem \ref{thm:altTwst}.

\begin{rem}
In addition to extending our recent paper \cite{GVSM}, these results should also be compared to Champanerkar and Kofman \cite{CK}, who prove similar results for strongly reduced weakly generalized alternating diagrams using a specialization of the Krushkal polynomial. 
Another similar result, Will was able to give two sided bounds for the twist number of a reduced alternating diagram in a thickened surface \cite{will2020homological}. However, this approach does not lead to a proof of the invariance of the twist number
\end{rem}

\section{Definitions}

\subsection{Generalized Link Projections}
We start first by reviewing the definition of a weakly generalized alternating link diagram, as defined by Howie and Purcell \cite{WGA}. Let $M$ be a compact, irreducible 3-manifold, $F\subset M$ a embedded, orientable, connected surface, and $\pi(L)\subset F$ a projection of a link $L$ onto $F$.

\begin{defn}
We say the \textit{representativity} of $\pi(L)$, denoted $r(\pi(L),F)$, is the minimum number of times a compression disk for $F$ intersects the diagram $\pi(L)$.

A diagram $\pi(L)$ is called \textit{weakly prime} if, whenever a disk $D\subset F$ has $\bndry D$ intersecting $\pi(L)$ transversely exactly twice, then either $F$ has positive genus and $D\cap \pi(L)$ is a single embedded arc, or $F=S^2$ and either $\pi(L)\cap D$ is a single embedded arc or $\pi(L)\cap (F\setminus D)$ is.
\end{defn}

With these two defintions in mind, we can fully define a weakly generalized alternating knot:

\begin{defn}
Let $M$ be a compact, irreducible 3-manifold with an embedded, orientable surface $F\subset M$. Then $\pi(L)\subset F$ is a \textit{weakly generalized alternating knot} if
\begin{enumerate}
	\item $\pi(L)$ is alternating on $F$,
	\item $\pi(L)$ is weakly prime,
	\item $\pi(L)$ has at least one crossing,
	\item $\pi(L)$ is checkerboard colorable, and
	\item $r(\pi(L),F)\ge 4$.
\end{enumerate}
\end{defn}

Note that, in the original paper \cite{GVSM}, it is also assumed that every essential simple closed curve $\gamma$ on $F$ that intersects $\pi(L)$ at exactly two points, one of the subarcs of $\pi(L)$ with endpoints on $\gamma$ has no crossings. We will not need that assumption for this paper, and can instead just work with weakly generalized alternating diagrams.

First, we recall the definition of a reduced knot diagram, as defined by Boden, Karimi, and Sikora \cite{ThickTaitConj}

\begin{defn}
A crossing of $\pi(L)\subset F$ is \textit{nugatory} if there is a a simple loop in $F$ which separates $F$ and interesects $\pi(L)$ only at the crossing. Such a crossing is \textit{removable} if the loop can be chosen to bound a disk.

A link diagram $\pi(L)$ on $F$ is \textit{reduced} if it has no removable nugatory crossings.
\end{defn}

When we start working with the coefficients, we will work with reduced alternating diagrams that have disk regions. Note, though, that all weakly generalized alternating diagrams must be reduced alternating as well; if not, we could find a loop that violates the weakly prime condition.

\subsection{Jones-like Polynomials}
We will use a modification to the Kauffman bracket function, and thus Jones-like polynomial, used in \cite{GVSM}. The bracket function is defined in the same way, except when we have a single unknot component in our diagram, we assign it $(-A^2-A^{-2})$. This will allow us to more easily work with states that have no contractible state circles. We can then, given a diagram $D$ on $F$, split $\langle D\rangle$ into several bracket polynomials based on the different Kauffman states. Recall that, given a Kauffman state $s$, $a(s)$ is the number of $A$ resolutions, $b(s)$ the number of $B$ resolutions, $s_t$ the collection of contractible state circles, and $s_{nt}$ the collection of non contractible circles. We will also denote the number of contractible state circles by $|s|_t$. Then we can define the following polynomials:

\begin{equation}
\label{eqn:altB0}
\langle D \rangle_0 = \sum_{\{s | s_{nt} = \emptyset\}} A^{a(s)-b(s)}(-A^2 - A^{-2})^{|s|_t}
\end{equation}

\begin{equation}
\label{eqn:altBX}
\langle D \rangle_X = \sum_{\{s | s_{nt} = X\}} A^{a(s)-b(s)}(-A^2-A^{-2})^{|s|_t}
\end{equation}

where $X\in X_F$ is a collection of simple closed disjoint curves on $F$. 
We can then define the Jones-like polynomials as follows:
\begin{align*}
J_0(\pi(L)) &= \left((-1)^{w(D)}A^{-3w(D)}\langle \pi(L)\rangle_0\right)|_{t=A^4} \\
J_X(\pi(L)) &= \left((-1)^{w(D)}A^{-3w(D)}\langle \pi(L)\rangle_X\right)|_{t=A^4}.
\end{align*}
We can write these out with coefficients:
\begin{align*}
J_0(\pi(L)) = a_{m,0}t^m + a_{m-1,0}t^{m-1} + &\cdots + b_{n+1,0}t^{n+1} + b_{n,0}t^n\\
J_X(\pi(L)) = a_{m,X}t^m + a_{m-1,X}t^{m-1} + &\cdots + b_{n+1,X}t^{n+1} + b_{n,0}t^n.
\end{align*}
As $J_0(\pi(L))$ and $J_X(\pi(L))$ might have different highest or lowest degrees, we take $m$ to be the highest degree over all possible $J_X(\pi(L))$ and $J_0(\pi(L))$, and $n$ to be the lowest. Then, if a polynomial has lower maximum degree, we take all higher coefficients to be 0.

Note the use of $\pi(L)$ in these definitions. In general, these are not link invariants, but instead diagram invariants.
Finally, we define
\[a_k = a_{k,0}+\sum_{X\in X_F}a_{k,X} \qquad b_k = b_{k,0} + \sum_{X\in X_F} b_{k,X}.\]

\subsection{Reduced Graphs}

Finally, we define certain graphs associated to $\pi(L)$, which we will use to connect the twist number to the coefficients.
\begin{defn}
Let $\pi(L)\subset F$ be a diagram for a link on a surface $L$. Let $s$ be some state for $\pi(L)$. Then the \textit{state graph} $G_s$ is defined as having vertices corresponding to the state circles of $s$, and edges connecting state circles that are adjacent across crossings of $\pi(L)$.

The \textit{reduced graph of} $G_s$, $G'_s$, can be obtained from $G_s$ by, whenever two edges of $G_s$ bound a disk on $F$, we remove one of them. Then $e'_a$ denotes the number of edges of $G'_s$.
\end{defn}

\section{Proofs}

We begin now by examining the coefficients for reduced alternating diagrams.

\begin{lem}
\label{lem:FirstCoeff}
Suppose $\pi(L)$ is twist-reduced, reduced alternating diagram on a closed orientable surface $F$, with $F\setminus \pi(L)$ disks. Suppose also that $\pi(L)$ is checkerboard colorable. Then $|a_m|=|b_n|=1$.
\end{lem}

\begin{proof}
As regions of $F\setminus \pi(L)$ are disks, the all $A$ state, $s_A$, has all state circles compressible disks (in particular, we can view $s_A$ as the collection of all regions of $F\setminus \pi(L)$ colored a single color). Then $s_A$ has highest degree $\left(c+2|s_{A}|_t\right)$, and the highest term will be $\pm 1$. 

Suppose we have a state $s_k$ with exactly $k$ $B$-resolutions, and a state $s_{k-1}$ which differs from $s_k$ by exactly one resolution change from an $A$-resolution to a $B$-resolution. This revolution change can, at most, increase the number of contractible state circles by 1. Look, then, at any state that has exactly one $B$ resolution, $s_1$. As $s_A$ has only contractible state circles, there are exactly three ways a resolution change from $s_A$ to a state $s_1$ can act.
\begin{enumerate}
	\item We could merge two contractible circles into a single contractible circle. Then $|s_{1}|_t = |s_{A}|_t-1$.
	\item We could split a single contractible circle into two noncontractible circles. Then $|s_{1}|_t = |s_{A}|_t-1$.
	\item We could split a single contractible circle into two contractible circles. Then $|s_{1}|_t = |s_{A}|_t+1$.
\end{enumerate}

Note that the highest degree $s_1$ can contribute to a polynomial is $ \left(c+2|s_1|_t-2\right)$. The only way $s_1$ could contribute to the highest degree of $J(\pi(L))$ is if the last case, (3), happens. However, note that for this to happen, the resolution change must occur at a crossing that meets a single state circle of $s_A$ twice. Then we can construct a loop $\gamma$ that intersects the diagram at exactly this crossing, and then follows along the boundary of the state circle. As there are no nugatory crossings, this must mean that $\gamma$ is a non-separating curve. In particular, $\gamma$ must be essential. After the resolution change, we will get two state circles that must have essential boundary, and thus be noncontractible. But then we must be in the second case, (2), and have a contradiction. So no $s_1$ can contribute to the highest degree.

Finally, look at any state $s_k$ with exactly $k$ $B$ resolutions. This must have highest degree $ \left(c+2|s_{k}|_t -2k\right)$. In order to contribute to the highest degree, we must have $|s_{k}|_t = |s_{A}|_t + k$. However, as each resolution change can, at most, introduce one contractible state circle, and we know that the change from $s_A$ to $s_1$ does not, we must have $s_k$ not contribute to the highest degree. So then $a_m = \pm 1$, and we are done.
\end{proof}

Before we begin, we introduce one last definition that we will use in the later half of the proof for the second coefficient.

\begin{defn}
Let $\pi(L)\subset F$ be a reduced alternating diagram, and $s$ a Kauffman state, with reduced graph $G'_s$. We say two edges, $\ell$ and $\ell'$ in $G'_s$ are \textit{parallel} if they are adjacent to the same vertices (or, equivalently, connect the same state circles) and, when placed on $F$ with the state circles, there is a disk in $F$ with boundary containing $\ell$, $\ell'$, and a portion of the state circles they connect.
\end{defn}

\begin{lem}
\label{lem:SecondCoeff}
Suppose $\pi(L)$ is a twist-reduced, reduced alternating diagram on a closed orientable surface $F$, with $F\setminus \pi(L)$ disks. Suppose also that $\pi(L)$ is checkerboard colorable. Then 
\begin{enumerate}
	\item $|a_{m-1}| = e'_A - |s_{A}|_t$.
	\item $|b_{n+1}| = e'_A - |s_{B}|_t$.
\end{enumerate}
\end{lem}

\begin{proof}
First, note that $s_A$ contributes to the second highest degree. In particular, it contributes $(-1)^{|s_A|_t}|s_A|_t$. Next, we also know that any state with exactly one $B$ resolution, $s_1$, will contribute to the second highest degree. We know this because we have one less contractible circle than $s_A$, by work above, one less $A$ resolution, and one more $B$ resolution; then we can use either Equation \ref{eqn:altB0} or \ref{eqn:altBX}. Now suppose we have a sequence of states $s_A, s_1, \cdots, s_k$, each differing from the previous state by a resolution change from an $A$-resolution to a $B$-resolution. Then, for $s_k$ to contribute, we must have each resolution change after the first introducing a new contractible state circle. We know that $s_1$ either has all contractible state circles, or has exactly two non-contractible state circles that are parallel to each other. Focus first on the states $s_1$ and $s_2$. Our goal is to show that $s_2$ will only contribute to $|a_{m-1}|$ if the edge changed from $s_1$ is parallel to the edge changed from $s_A$.

There are three possible ways a resolution change can act and still have $s_2$ contribute to the second highest degree:
\begin{enumerate}
	\item We split a single contractible circle into two contractible circles.
	\item We merge two noncontractible circles into a single contractible circle.
	\item We split a single noncontractible circle into a noncontractible circle and a contractible circle.
\end{enumerate}

We focus on each of these cases individually. In the case of (1), where we split a contractible circle into two contractible circles and don't change the pattern of noncontractible state circles, the edge changed, $\ell_2$ would have to, in $s_1$, be adjacent to a single state circle. If this edge is parallel to the edge changed from $s_A$ to $s_1$, $\ell_1$, we are fine. Otherwise, we must have both $\ell_2$ and $\ell_1$ appear in $G'_A$, and so $\ell_1$ and $\ell_2$, with a portion of some state circles, forms an essential curve on $F$. Then $s_2$ will introduce a noncontractible circle, and so we cannot be in the first case.

In the case of (2), where we merge two noncontractible circles into a single contractible circle, denote the edge changed in the resolution change from $s_A$ to $s_1$ as $\ell_1$, and the edge changed from $s_1$ to $s_2$ as $\ell_2$. We know that the two noncontractible circles must be parallel, and so, to merge them to get a contractible circle, we must have $\ell_1$ parallel to $\ell_2$. Look at $\ell_2$ in $s_A$, and construct a new state $s'_1$, which has all $A$-resolutions except at $\ell_2$. In Figure \ref{fig:NonContractible}, the leftmost state is $s_1$, the center $s_A$, and the right $s'_1$.

\begin{figure}[!h]
	\centering
		\def\svgwidth{.5\columnwidth}
	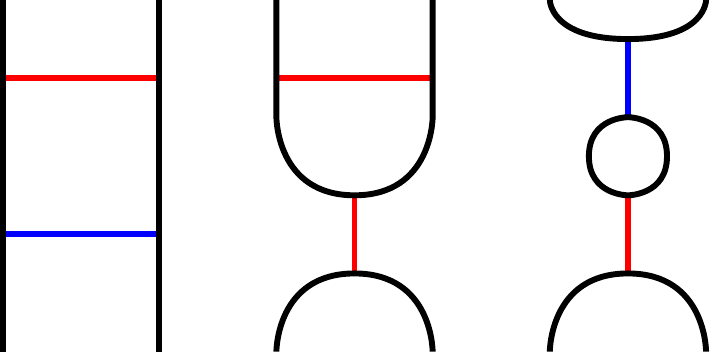
	\caption{
		\textit{(Left)} The state $s_1$ gotten from $s_A$ by merging two noncontractible circles into a single contractible circle;
		\textit{(Center)} The all $A$ state, $s_A$;
		\textit{(Right)} The modified state, $s'_1$, with exactly one $A$ resolution.
	}
	\label{fig:NonContractible}
\end{figure}

Note that $s'_1$ must, then, introduce a new contractible state circle from $s_A$. In addition, it will have the same types of non-contractible state circles as $s_A$. Then, by Boden, Karimi, and Sikora \cite[Theorem 11]{ThickTaitConj}, any state of a reduced alternating diagram (and thus a weakly generalized alternating diagram) with one $B$ resolution must have, at most, the same number of contractible state circles as $s_A$, so we get a contradiction.

Finally, in the case of (3), we split a noncontractible state into a contractible circle and a noncontractible circle. Label the edges as before: $\ell_1$ is the edge first changed in the resolution change from $s_A$ to $s_1$, and $\ell_2$ the edge changed from $s_1$ to $s_2$. There are two possibilities for $\ell_2$: either it is parallel to $\ell_1$ in $s_A$, or it is not. If the two edges are not parallel, then $\ell_2$ must have it's endpoints both to the same side of $\ell_1$, and cut off a contractible circle disjoint from $\ell_1$. Then we can construct a new state, $s'_1$ that has all $A$ resolutions except for at $\ell_2$, where we have a $B$ resolution. As $\ell_2$ cuts off a contractible circle disjoint from $\ell_1$, $s'_1$ must have more contractible circles than $s_A$, a contradiction. So then $\ell_1$ and $\ell_2$ must be parallel in $s_A$.

Note that, in order for $s_2$ to contribute to the second highest degree, the second edge must be parallel to the first. If it isn't, it either does not introduce a new contractible circle (like in the first case), or we can find a state with 1 B resolution that has more contractible circles than $s_A$. We can do this same process for any state $s_k$ with $k$ B resolutions, and so all states that contribute to $|a_{m-1}|$ must have $B$ resolutions at only parallel edges.

Also note, if any two edges are parallel, they must be part of the same twist region. To see this, as parallel edges bound a disk on $F$, we get a disk whose boundary intersects the link diagram at exactly two crossings. As the diagram is twist reduced, this must mean both of these crossings, and thus the edges, belong to the same twist region.

So then we get a contribution to $|a_{m-1}|$ only when we make resolution changes to edges that, in $s_A$, are all parallel to each other. As we are asking about $|a_{m-1}|$, we will leave off the essential curve variable $X$ during our calculations. Each such state will contribute $(-1)^{|s_A|-2}(-1)^j$, where $j$ is the number of changed edges, and there are $\binom{k}{j}$ possible states with exactly $j$ resolution changes, where $k$ is the total number of parallel edges. Using the binomial theorem, we get that each family of parallel edges must contribute $(-1)^{|s_A|-1}$ to the second highest coefficient. As each family of parallel edges corresponds to a single edge in $G'_A$, and factoring in the contribution to the second coefficient from $s_A$, we get that
\begin{align*}
|a_{m-1}| &= | (-1)^{|s_A|_t-1} e'_A + (-1)^{|s_A|_t}(|s_A|_t)| \\
		  &= |(-1)^{|s_A|_t-1} ( e'_A - |s_A|_t)|\\
		  &= e'_A - |s_A|_t
\end{align*}
and so we are done with part (1) of the theorem.

To prove part (2), we can apply the same argument to the diagram $D^*$.

\end{proof}

We can now use Lemmas \ref{lem:FirstCoeff} and \ref{lem:SecondCoeff} to prove one of our main results:

\begin{thm}
\label{thm:TwistNumber}
Suppose that $\pi(L)$ is a twist-reduced, reduced alternating link projeciton with twist number $t_F(\pi(L))$ on a projection surface $F\subset M$ of genus at least 1.
Suppose also that the regions of $F\setminus \pi(L)$ are all disks. Then
\[|a_{m-1}| + |b_{n+1}| = t_F(\pi(L)) - \chi(F).\]
\end{thm}

\begin{proof}
 First, note that we have

\[t_F(\pi(L)) = c - (c-e'_B) - (c-e'_B) = e'_A + e'_B - c.\]

Next, as the regions of $F\setminus \pi(L)$ are disks, we also have
\[|s_A|+|s_B| = c + 2 - 2g(F) = c + \chi(F).\]
It is important to note that, as $F\setminus \pi(L)$ are all disks, then $|s_A| = |s_{A}|_t$ and $|s_B| = |s_{B}|_t$, as there are only contractible circles. Then we get
\begin{align*}
|a_{m-1}| + |b_{n+1}| &= e'_A - |s_{A}|_t + e'_B - |s_{B}|_t\\
					  &= e'_A + e'_B - |s_A| - |s_B|\\
					  &= t_F(\pi(L)) + c - c - \chi(F)\\
					  &= t_F(\pi(L)) - \chi(F),
\end{align*}
and we are done.
\end{proof}

The proofs for Theorems \ref{thm:altVol} and \ref{thm:altTwst} follow their original proofs in \cite{GVSM}, with no modifications except for the new numbers. We will sketch their proofs here for completeness.

\begin{proof}[(Proof of \ref{thm:altTwst})]
We have already proved the twist number result above, in Theorem \ref{thm:TwistNumber}, as well as the last last statement about reduced alternating diagrams. So we now focus on the other two statements.

Let $S_A$ and $S_B$ be the checkerboard surfaces of $\pi(L)$, $X= M\setminus L$, and $M_A = X\cut S_A$ and $M_B = X\cut S_B$. As $\pi(L)$ is a weakly generalized alternating diagram, \cite{WGA} gives us
\begin{equation}
\label{eqn:guts}
\chi(\guts(M_A)) = \chi(F) + \frac{1}{2}\chi(\bndry M) - |s'_B|,
\end{equation}
where $|s'_B|$ is the number of non-bigon $B$ regions of $S_A$. Then note that
\[ \chi(F) = |s_A| - e'_A + |s'_B|.\]
In particular, by rearranging, we get
\[\chi(F) - |s'_B| = |s_A| - e'_A\]
Finally, by Lemma \ref{lem:SecondCoeff}, we get
\begin{equation}
\label{eqn:EulerChar}
\chi(F) - |s'_B| = -|a_{m-1}|.
\end{equation}
Then by combining Equations \ref{eqn:guts} and \ref{eqn:EulerChar}, we immediately get our result. The proof for $\guts(M_B)$ follows similarly, by swapping $A$ and $B$.

\end{proof}

\begin{proof}[(Proof of \ref{thm:altVol})]
Let $S_A$ and $S_B$ be the checkerboard surfaces of $\pi(L)$, $X= M\setminus L$, and $M_A = X\cut S_A$ and $M_B = X\cut S_B$. Then, by \cite{WGA}[Theorem 9.1], we have
\[ \vol(M) \ge -v_8 \chi(\guts(M_A)) \quad \mathrm{and} \quad \vol(M) \ge -v_8 \chi(\guts(M_B)).\]
Then, by \ref{thm:altTwst}, we immediately get our result.

\end{proof}

\bibliographystyle{amsplain}
\bibliography{ShortSMA}
\end{document}

%% file: pics/2ntc.pdf_tex
\begingroup%
  \makeatletter%
  \providecommand\color[2][]{%
    \errmessage{(Inkscape) Color is used for the text in Inkscape, but the package 'color.sty' is not loaded}%
    \renewcommand\color[2][]{}%
  }%
  \providecommand\transparent[1]{%
    \errmessage{(Inkscape) Transparency is used (non-zero) for the text in Inkscape, but the package 'transparent.sty' is not loaded}%
    \renewcommand\transparent[1]{}%
  }%
  \providecommand\rotatebox[2]{#2}%
  \newcommand*\fsize{\dimexpr\f@size pt\relax}%
  \newcommand*\lineheight[1]{\fontsize{\fsize}{#1\fsize}\selectfont}%
  \ifx\svgwidth\undefined%
    \setlength{\unitlength}{340.33464567bp}%
    \ifx\svgscale\undefined%
      \relax%
    \else%
      \setlength{\unitlength}{\unitlength * \real{\svgscale}}%
    \fi%
  \else%
    \setlength{\unitlength}{\svgwidth}%
  \fi%
  \global\let\svgwidth\undefined%
  \global\let\svgscale\undefined%
  \makeatother%
  \begin{picture}(1,0.4958355)%
    \lineheight{1}%
    \setlength\tabcolsep{0pt}%
    \put(0,0){\includegraphics[width=\unitlength,page=1]{pics/2ntc.pdf}}%
    \put(0.04383128,0.3966687){\color[rgb]{0,0,0}\makebox(0,0)[lt]{\lineheight{1.25}\smash{\begin{tabular}[t]{l}$\ell_2$\end{tabular}}}}%
    \put(0.0416276,0.17629773){\color[rgb]{0,0,0}\makebox(0,0)[lt]{\lineheight{1.25}\smash{\begin{tabular}[t]{l}$\ell_1$\end{tabular}}}}%
  \end{picture}%
\endgroup%